\documentclass[10]{amsart}
\usepackage{amssymb}
\usepackage{amsmath}
\usepackage{hyperref}
\usepackage{color}
\newcommand{\bt}{\begin{theorem}}
\newcommand{\et}{\end{theorem}}
\newcommand{\bi}{\begin{itemize}}
\newcommand{\ei}{\end{itemize}}
\newcommand{\bea}{\begin{eqnarray}}
\newcommand{\ba}{\begin{array}}
\newcommand{\eea}{\end{eqnarray}}
\newcommand{\ea}{\end{array}}
\newcommand{\be}{\begin{equation}}
\newcommand{\ee}{\end{equation}}

\newcommand{\what}{\widehat}%
\newcommand{\wtilde}{\widetilde}%
\newcommand{\R}{\mathbb R}%
\newcommand{\B}{\mathbb B}%
\newcommand{\C}{\mathbb C}%
\newcommand{\T}{\mathbb T}%
\newcommand{\Q}{\mathbb Q}%
\newcommand{\N}{\mathbb N}%
\newcommand{\hc}{\mathrm c}
\newtheorem{theorem}{Theorem}[section]

\newtheorem{lemma}[theorem]{Lemma}
\newtheorem{proposition}[theorem]{Proposition}
\newtheorem{corollary}[theorem]{Corollary}

\theoremstyle{definition}

\theoremstyle{definition}
\newtheorem{remark}[theorem]{Remark}

\newtheorem{example}[theorem]{Example}
\numberwithin{equation}{subsection}
\numberwithin{theorem}{subsection}

\begin{document}
\baselineskip18pt
\author[S. K. Ray]{Swagato K. Ray }
\address[S. K. Ray]{Stat-Math Unit, Indian Statistical
Institute, 203 B. T. Rd., Calcutta 700108, India}
\email{swagato@isical.ac.in}

\author[R. P. Sarkar]{Rudra P. Sarkar}
\address[R. P. Sarkar]{Stat-Math Unit, Indian Statistical
Institute, 203 B. T. Rd., Calcutta 700108, India}
\email{rudra@isical.ac.in}


\title[Chaotic  multipliers]{Chaotic behaviour of the Fourier multipliers on Riemannian symmetric spaces of noncompact type}
\maketitle
\begin{abstract} Let $X$ be a Riemannian symmetric space of noncompact type and $T$ be a  linear translation-invariant operator which is bounded on $L^p(X)$. We shall show that if $T$ is not a constant multiple of identity then there exist complex constants $z$ such that $zT$ is chaotic on $L^p(X)$ when $p$ is in  the  sharp range $2<p<\infty$.  This vastly generalizes the result that dynamics of the (perturbed) heat semigroup is chaotic on  $X$ proved in \cite{J-W, Pram-Sar}.  \end{abstract}

\section{Introduction} Let $T:L^p(X)\to L^p(X)$ be an $L^p$-multiplier on a Riemannian symmetric space $X$ of noncompact type for any fixed $1\le p\le \infty$. We shall call $T$ nontrivial if it is not a constant multiple of identity.  The aim of this note is  to show that for the range $2<p<\infty$, given any  nontrivial $L^p$-multiplier $T$, we can find  complex constants $z\in \C$ such that the operator $zT$  is chaotic on $L^p(X)$.  This range of $p$ will be shown to be sharp. Our definition of chaos is consistent with \cite{DSW, J-W}, which in turn is an adaptation of the one introduced by Devaney \cite{Devaney}.

To put the result in perspective, let us discuss the background. Let $\Delta$ be the positive Laplace-Beltrami operator on $X$ and $T_t=e^{-t\Delta}, t\ge 0$ be the heat semigroup. In \cite{J-W} Ji and Weber had shown that its  perturbation  $e^{ct}T_t, t\ge 0$ for some constant $c$ is subspace chaotic on $L^p(X)$  when $2<p<\infty$. Through \cite{Sar-NA-chaos}  and \cite{Pram-Sar} this result was improved by establishing that the same perturbation of the heat semigroup is actually  chaotic on $L^p(X)$ with $p$ in the same range. In this paper we shall establish that this is a particular case of a  general fact. We first note  that the operator $T_t=e^{-t\Delta}$ is the same as the operator $f\mapsto f\ast h_t$ where $h_t$ is the heat-kernel, i.e. the fundamental solution of the heat equation $(\Delta-\frac{\partial}{\partial t})f=0$. Thus it is natural to consider  the operator $f\mapsto f\ast \mu$ where $\mu$ is any nonatomic $K$-biinvariant Borel measure, a particular case of which is the heat operator. We show that such an operator is always chaotic on $L^p(X), 2<p<\infty$ provided it is not a contraction (Corollary \ref{result-measure}). Indeed the heat operator can be substituted by any $L^p$-multiplier. A corollary of the main result (Theorem \ref{result-multiplier}) in this paper is the following. Consider an {\em autonomous discretization} (see section 2.2) of the heat semigroup $T_{t_0}=e^{-t_0\Delta}$  for any fixed $t_0>0$. Then   there exists constant $z\in \C$, such that   $T=ze^{-t_0\Delta }$  is chaotic on $L^p(X)$ when $p$ is in the range $2<p<\infty$.
 It is known that a $C_0$-semigroup is hypercyclic if and only if it admits a hypercyclic discretization. We also note that a periodic point of a discretizatized semigroup is also a periodic point of the original semigroup. Thus if $T=zT_{t_0}$ is chaotic then so is the semigroup $(zT_t)_{t\ge 0}$. Thus   our result in the present article accommodates the earlier results in this direction mentioned above. See Section 2 and Section 5 for more details.

The paper is organized as follows.
The general preliminaries are established in Section 2, while that about Riemannian symmetric spaces are given in Section 3. Section 4 contains the main result and its proof. In Section 5 we deal with some well-known multipliers and obtain some corollaries of the main results for the particular cases.  In Section 6  we show the sharpness of the range of $p$ in the main result. Finally in Section 7 we state some open questions along with motivations.

\section{Preliminaries} \label{prelim}
In this section we shall establish notation and gather all the definitions and results required for this article.
\subsection{Generalities} \label{generalities}
The letters $\mathbb R$, $\mathbb Q$ and $\mathbb C$ denote respectively the set of real numbers, rational numbers and complex numbers.
 We use $\Re z$ and $\Im z$ to denote respectively the real and imaginary parts of $z\in \C$. This notation will also be used for its obvious generalization when $\mathbf z \in \mathbb C^n$.
 The notation $| \cdot |$ will denote the standard Euclidean norm in $\mathbb R^n$ and in $\mathbb C^n$:
$|\mathbf x| = \sqrt{x_1^2 + \cdots + x_n^2}$ for  $\mathbf x = (x_1, \cdots, x_n) \in \mathbb R^n$ and   $|\mathbf z| = \sqrt{|\Re(\mathbf z)|^2 + |\Im(\mathbf z)|^2}$ for $\mathbf z \in \mathbb C^n$.
 We will also use $|\cdot|$ to represent a norm on certain spaces related to the symmetric space $X$, but under appropriate identifications of these spaces with $\mathbb R^n$ or $\mathbb C^n$, which will make this  consistent with the previous usage of this notation. (For more details see  Section \ref{RSS-background}.)
 For a set $S$ in a topological space, $S^{\circ}$ denotes its interior.
 For a function $f$ on $X$, $\|f\|_p$ denotes its $L^p$ norm. We shall mention explicitly when we will use the $L^p$-norm of functions on spaces other than $X$.
 For any $p\in (1, \infty)$, $p'=p/(p-1)$ and for $p=1$, $p'=\infty$. When $p=\infty$ we use $p'$ to mean $1$.
We shall frequently use the  notation
$\gamma_p = \gamma_{p'} = \Bigl|\frac{2}{p}-1\Bigr|$  for any $p \in (1, \infty)$ and   $\gamma_1 = \gamma_{\infty} = 1$. The letters $C, c$ will be used to denote  positive constants whose values may change from one line to another.
Following results of several complex variable will be used.
\subsubsection{Open mapping theorem} (\cite[Theorem 1.21, p. 17]{Range})  If $\Omega\subset\C^n$ is open and $f:\Omega\rightarrow\C$ is a nonconstant holomorphic function then $f(U)$ is open for every open set $U\subset\Omega$.

\subsubsection{Maximum modulus principle} (\cite[Corollary 1.22, p. 17]{Range}) Let $\Omega\subset\C^n$ be an open set and $f:\Omega\rightarrow\C$ is a holomorphic function. If $|f|$ attains a local maximum at a point $z_0\in \Omega$ then $f$ is constant in the connected component of $\Omega$ containing $z_0$.

\subsubsection{Thin sets} Let $\Omega\subset\C^n$ be an open set. A subset $E$ of $\Omega$ is called  {\em thin} if for every point $x_0\in \Omega$ there is a ball $B(x_0, r)$ centered at $x_0$ with radius $r>0$ in $\Omega$ and a nonconstant holomorphic function $f: B(x_0, r)\to \C$ such that $f(z)=0$ for $z\in E\cap B(x_0, r)$.  We quote here some well known results related to thin sets (\cite[pp. 32--33]{Range}):
  \begin{enumerate}
  \item If $E\subset \Omega$ is not thin then no nonzero holomorphic function $f:\Omega\to \C$ can vanish on $E$.
  \item If $E\subset \Omega$ is thin then its closure $\bar{E}$ in $\Omega$ is also thin and $E$ is nowhere dense.
  \item  The $2n$-dimensional Lebesgue measure of  a thin set $E\subset \Omega\subset \C^n$ is zero.
  \item If $\Omega$ is connected and $E\subset \Omega$ is thin then $\Omega\smallsetminus E$ is also connected.
    \end{enumerate}

\subsection{Chaos, Hypercyclicity etc.}  Let $\B$ be  a separable Fr\'echet space and  $T:\B\to \B$ be a linear dynamical system, i.e. $T$ is a linear map from  $\B$ to itself. For $x\in \B$ we call \[\{x, Tx, T^2x, \ldots, \}\] the orbit of $x$ under $T$. The operator $T$ is called hypercyclic if there is an $x\in \B$, such that the orbit  of $x$ under $T$ is dense in $\B$. In such a case $x$ is called a hypercyclic vector for $T$. (See \cite[p. 37]{GKPA}.)
A point $x\in \B$ is called a {\em periodic point} of $T$ if there is a nonzero natural number $n$ such that $T^nx=x$.
The operator $T$ is called chaotic if $T$ is hypercyclic  and the set of all its  periodic points is dense in $\B$.

For a $C_0$-semigroup $(T_t)_{t\ge 0}$   on a Fr\'echet space $\B$, and $x\in \B$, $\{T_t x \mid t\ge 0\}$ is called the orbit of $x$ under $(T_t)_{t\ge 0}$. If this orbit is dense in $\B$, then $x$ is called a hypercyclic vector and we say that $(T_t)_{t\ge 0}$ is hypercyclic on $\B$. A point $x\in \B$ is called a periodic point of $(T_t)_{t\ge 0}$ if $T_tx=x$  for some $t>0$. The semigroup $(T_t)_{t\ge 0}$ is called chaotic
 if it is hypercyclic  and the set of all its  periodic points is dense in $\B$.

 A {\em discretization} of a $C_0$-semigroup $(T_t)_{t\ge 0}$ is a sequence of operators $(T_{t_n})_n$ with $t_n\to \infty$. In particular if $t_n=nt_0$ for some $t_0>0$ and $n\in \N$, then
 $(T_{t_n})_n=(T^n_{t_0})_n$ is an {\em autonomus dscretization} of $(T_t)_{t\ge 0}$. It is clear that any periodic point (respectively hypercyclic vector) of the operator $T_{t_0}$ for any fixed $t_0>0$ is also  a periodic point (respectively hypercyclic vector) of the semigroup $(T_t)_{t\ge 0}$. Thus if for some $t_0>0$, $T_{t_0}$ is chaotic on a Banach space $\B$,  then $(T_t)_{t\ge 0}$ is chaotic on $\B$.  We also have the following result (see \cite[p. 168, Theorem 6.8]{GKPA}):
 \begin{proposition} Let $(T_t)_{t\ge 0}$ be a $C_0$-semigroup on a Banach space $\B$. If $x\in \B$ is a hypercyclic vector for $(T_t)_{t\ge 0}$ then it is a hypercyclic vector for each operator $T_t$, $t>0$.
 \end{proposition} For a detailed account on the relationship between dynamics of a $C_0$-semigroup and that of its discretization we refer to \cite[Chap. 7]{GKPA}. This discussion in particular points out that  Theorem \ref{result-multiplier} in this paper accommodates the chaoticity of the heat semigroup considered in \cite{J-W, Sar-NA-chaos, Pram-Sar} as a special case, which was alluded to in the introduction.

The following result due to Kitai   will be used in this article (\cite{Kitai}, \cite[p. 71]{GKPA}).
\begin{theorem}[Kitai]\label{kitai} Let $\B$ be a separable Banach space and $T$ be a bounded linear operator from $\B$ to itself. Let $Y_1, Y_2$ be two dense subsets of $\B$ and $T':Y_1\to Y_1$ be a  (not necessarily linear or continuous) map. If
\[(i)\ \  \lim_{n\to \infty}T^n y =0 \ \ \forall y\in Y_2, (ii) \ \ \lim_{n\to \infty} T'^nx=0 \ \ \forall x\in Y_1, \text{ and } (iii)\ \  TT'x=x\ \ \forall x\in Y_1\] then $T$ is  hypercyclic on $\B$., i.e there is an $x\in \B$ such that the orbit $\{T^nx\mid n\in \N\}$ is dense in $\B$.
\end{theorem}
We conclude this section noting that  (\cite[p. 167, Theorem 6.7]{GKPA}) if $T$ is a hypercyclic operator on a Fr\'echet space $\B$ and $\lambda\in \C$ with $|\lambda|=1$ then $T$ and $\lambda T$ have the same set of hypercyclic vectors.

\section{Riemannian symmetric spaces} \label{RSS-background}
Most of the notation and results in this section are standard and available  for instance in  \cite{GV, Helga3}. For convenience and the sake of keeping the current exposition self-contained, we merely collect the relevant facts without proofs but indicate appropriate references.

\subsection{Basics} \label{def-rho}
Throughout this paper, $X$ will denote a Riemannian symmetric space of noncompact type which  can be realized as a quotient space $G/K$ where $G$ is a connected noncompact semisimple Lie group with finite centre, and $K$ is a maximal compact subgroup of $G$.  The group $G$ acts naturally on $X$ and on functions on $X$  by left translations. Functions on $X$ are identified with the right $K$-invariant functions on $G$ and vice versa.  For an element $x\in G$ and a function $f$ on $X$, $\ell_x f$ is the left translation of $f$ defined by $\ell_x f(y)=f(x^{-1}y)$. A function (or measure) on $X$ is called  $K$-invariant, if it is invariant under left $K$-action. Such a function (respectively  measure) can be identified naturally with a $K$-biinvariant function (respectively  measure) on $G$. Frequently we shall use this identification without mentioning.

 The group $G$ admits an Iwasawa decomposition, namely $G = KAN$, inducing a direct sum decomposition of the  Lie algebra: $\mathfrak g = \mathfrak k \oplus \mathfrak a \oplus \mathfrak n$. Here $\mathfrak g, \mathfrak k, \mathfrak a$ and $\mathfrak n$ denote the Lie algebras of $G, K, A$ and $N$ respectively. This decomposition fixes a system of positive roots $\Sigma^{+} \subset \mathfrak a^{\ast}$, where $\mathfrak a^{\ast}$ denotes the real dual of $\mathfrak a$. From the collection  of root-spaces  $\mathfrak g_{\alpha}$, parametrized by $\Sigma^{+}$  one obtains \[ \mathfrak n = \bigoplus_{\alpha \in \Sigma^{+} }\mathfrak g_{\alpha}. \]   Setting $m_{\alpha} = \text{dim}(\mathfrak g_{\alpha})$, the multiplicity of the root $\alpha \in \Sigma^{+}$, we define  ${\pmb{\rho}}$  as the half-sum of the elements of $\Sigma^{+}$ counted with multiplicities:
\begin{equation} \label{rho}
{\pmb{\rho}} = \frac{1}{2} \sum_{\alpha \in \Sigma^{+}} m_{\alpha} \cdot \alpha \in \mathfrak a^{\ast}.
\end{equation}
The Killing form on $\mathfrak g$ restricts to a positive definite form on $\mathfrak a$, which in turn induces a positive inner product and hence a norm $|\cdot|$ on $\mathfrak a^{\ast}$, so $|{\pmb{\rho}}|$ is defined. The Killing form endows $X$ with both a natural Riemannian metric and a corresponding $G$-invariant measure (denoted  by $dx$). The positive Laplace-Beltrami operator corresponding to this Riemannian metric is denoted by $\Delta$.

Let $\dim(\mathfrak a) = n$, which is by definition the rank of the space $X$. Using the pull-back of the Killing form, we will henceforth identify $\mathfrak a$ and $\mathfrak a^{\ast}$ with $\mathbb R^n$, equipped with the standard inner product $\langle \cdot, \cdot \rangle$:
\begin{equation} \label{ip}\langle \mathbf x, \mathbf y \rangle = \sum_{i=1}^{n} x_i y_i, \quad \mathbf x = (x_1, \cdots, x_n), \; \mathbf y = (y_1, \cdots, y_n) \in \mathbb R^n, \end{equation}
so that $\langle \mathbf x, \mathbf x \rangle > 0$ for all nonzero $\mathbf x \in \mathbb R^n$.  The complexification of $\mathfrak a^{\ast}$ will be denoted by $\mathfrak a_{\mathbb C}^{\ast}$ and will be naturally identified with $\mathbb C^n$. The real inner product \eqref{ip} extends to $\mathbb C^n$ as a $\mathbb C$-bilinear form $\mathbb C^n \times \mathbb C^n \rightarrow \mathbb C$ defined by
\begin{equation} \label{def-L}
 (\mathbf z, \mathbf v) = \sum_{i=1}^{n} z_i v_i, \quad  \text{ where } \mathbf z = (z_1, \cdots, z_n), \, \mathbf v=(v_1, \cdots, v_n) \in \mathbb C^n.
\end{equation} For the action of  $\lambda\in \mathfrak a_{\mathbb C}^{\ast}$ on $v\in \mathfrak a$ we shall use both the notation $\lambda(v)$ and $(\lambda, v)$.

Let $W$ denote the  Weyl group  of the pair $(\mathfrak g, \mathfrak a)$ and  $\mathfrak a_+$ and $\mathfrak a^\ast_+$ be the positive Weyl chambers corresponding to $\Sigma^+$ in $\mathfrak a$ and $\mathfrak a^\ast$ respectively.

 For $p \geq 1$, we define  the set \cite[p.328]{GV}:
\begin{equation} \label{Lambdap}
\Lambda_p = \left\{ \lambda  \in \mathbb C^n \mid |\Im (w\lambda)(H)|\le \gamma_p\, \pmb\rho(H)\text{ for all } H\in \mathfrak a_+, w\in W \right\},
\end{equation}
where $\gamma_p$ and ${\pmb{\rho}}$ have been defined in \ref{generalities} and \eqref{rho} respectively. We note:
 \begin{enumerate}
 \item[(a)] If $p=2$ then $\Lambda_p$ reduces to $\mathfrak a^{\ast}$, which is identified with $\R^n$.
 \item[(b)] For $1\le p<q\le 2$, $\Lambda_q\subsetneq \Lambda_p$.
 \item[(c)] $\Lambda_p=\Lambda_{p'}$ for $p\ge 1$.
 \item[(d)] $\Lambda_p$ is closed under the reflection $\lambda\mapsto -\lambda$ (\cite[p. 329]{GV}).
\end{enumerate}

We recall from subsection \ref{def-rho} that the $G$-invariant measure $dx$ on $X$ is induced by the Killing form. On $G$, we fix the Haar measure $dg$ that satisfies
\[ \int_X f(x) \, dx = \int_G f(g) \, dg\]
for every function $f\in L^1(X)$ which is identified as  a right $K$-invariant function on $G$ in the right hand side. Let $M$ be the centralizer of $A$ in $K$. On $K$ we fix the normalized Haar measure $dk$ and on $K/M$ we fix the $K$-invariant normalized measure. We shall often slur over the difference between the two.

\subsection{Spherical Fourier Transform}
Let $H : G \rightarrow \mathfrak a$ be the Iwasawa projection associated to the decomposition $G = KAN$. The {\em{elementary spherical function}}  $\varphi_\lambda$ for $\lambda\in \mathfrak a^\ast_\mathbb C$ is defined by (\cite[p. 200]{Helga3}):
\[\varphi_\lambda(x)=\int_K e^{- (i\lambda+{\pmb{\rho}}) (H(x^{-1}k))}dk, \quad x\in G.\]
We record a few well-known facts about these functions. Some are easy to deduce. For the others, see  \cite[pp. 419, 427, 460]{Helga2} and \cite[Proposition 3.1.4., Proposition 3.2.2., (4.6.3), (4.6.4), (4.6.9)]{GV}.
\begin{lemma}  \label{lemma-eigenfns}
The elementary spherical functions $\varphi_\lambda$ have the following properties:
\begin{enumerate}
\item[(a)] For each $\lambda\in \mathfrak a^\ast_\C\equiv \C^n$, the function $\varphi_\lambda$ is a $K$-biinvariant function on $G$
(hence is naturally identified as a function on $X$) and  $\int_{K} \varphi(xky)dk=\varphi(x)\varphi(y)$.
\item[(b)]  $\varphi_\lambda = \varphi_{w \lambda}$ for all $w \in W$.
\item[(c)] $\varphi_{-\lambda} (x^{-1}) = \varphi_\lambda(x)$ for all $x \in G$ and $\lambda \in \mathfrak a_{\mathbb C}^{\ast}$.
\item[(d)] For every $\lambda$, the identity \[ \Delta \varphi_\lambda = ((\lambda, \lambda) + |{\pmb{\rho}}|^2) \varphi_\lambda \]
holds pointwise.
\item[(e)] If $2 < p < \infty$ and $\lambda \in \Lambda_p^{\circ}$ then $\varphi_\lambda \in L^p(X)$ and for $\lambda\in \mathfrak a^\ast$, $\varphi_\lambda\in L^{2+\epsilon}(X)$ for any $\epsilon>0$. \label{eigenfns-a}
\item[(f)] If $\lambda\in \Lambda_1$, then $\varphi_\lambda \in L^\infty(X)$. \label{eigenfns-b}
\item[(g)] If $\lambda \in \Lambda_1^{\circ}$, then $\varphi_\lambda \in C_0(X)$, the space of continuous functions vanishing at infinity. \label{eigenfns-c}
\item[ (h)] For each fixed $x\in G$, $\lambda\mapsto \varphi_\lambda (x)$ is a holomorphic function on $\C^n$.
\end{enumerate}
\end{lemma}

For a measurable function $f$ of $X$, we define its {\em{spherical Fourier transform}} $\widehat{f}$ as follows (see \cite[p. 425]{Helga2}),
\[\widehat{f}(\lambda) = \int_X f(x) \varphi_{-\lambda} (x) \, dx, \quad \lambda \in \mathfrak a^{\ast}, \]
whenever the integral makes sense.  Since for all $w \in W$, $\varphi_\lambda = \varphi_{w \lambda}$ we have $\what{f}(\lambda)=\what{f}(w\lambda)$. Its inverse transform, again subject to convergence of the defining integral, is given by (see \cite[p. 454]{Helga2})
\begin{equation} \label{InvSpTr}
f(x) = C \int_{\mathfrak a^{\ast}} \widehat{f}(\lambda) \, \varphi_\lambda(x) \, |\hc(\lambda)|^{-2} d\lambda,
\end{equation}
where $\hc(\lambda)$ is the Harish-Chandra $\hc$-function, $d\lambda$ is the Lebesgue measure on $\mathfrak a^\ast$ (and thus $|\hc(\lambda)|^{-2} d\lambda$ is the spherical Plancherel measure on $\mathfrak a^{\ast}$) and $C$ is a normalizing constant.

 \subsection{Helgason Fourier transform} (See \cite[pp. 199-203]{Helga3} for details.) For a function  $f$ on $X$,  its Helgason Fourier transform is defined by \[\wtilde{f}(\xi, k)=\int_X f(x) e^{(i\xi-\pmb\rho)(H(x^{-1}k))} dx\] for all $\xi\in \mathfrak a_\C^\ast \equiv \C^n$,  $k\in K/M$. for which the integral exists. The Fourier transform $f(x)\to \wtilde{f}(\xi, k)$ extends to an isometry of $L^2(X)$ onto $L^2(\mathfrak a^\ast_+\times K/M, |\hc(\xi)|^{-2} d\xi dk)$ where $\hc(\xi)$ is the Harish-Chandra $\hc$-function and thus $|\hc(\xi)|^{-2} d\xi dk$ is the Plancherel measure.  We also have,
 \[\int_X f_1(x) \overline{f_2(x)} dx=\frac 1{|W|} \int_{a^\ast_+\times K/M} \wtilde{f_1}(\xi, k)\overline{\wtilde{f_2}(\xi, k)} |\hc(\xi)|^{-2} d\xi dk,\] where $|W|$ is the cardinality of the Weyl group $W$ and $dk$ is the normalized $K$-invariant measure on $K/M$.
 We note that if  $g$ is a   $K$-invariant function on $X$, then $\wtilde{g}(\xi, k)=\what{g}(\xi)$ for all $k\in K/M$ and for $f, g$ as above,
\[\wtilde{f\ast g}(\xi, k)=\wtilde{f}(\xi, k)\what{g}(\xi)\] for $\xi\in \C^n$ and $k\in K/M$ whenever the quantities $f\ast g, \wtilde{f\ast g}, \wtilde{f}$ and $\what{g}$ make sense. We have the following $L^p$-version of the inversion formula (see \cite[3.3]{St-Tom79}).
\begin{theorem}[Stanton-Tomas] \label{St-Tom79}
For a function $f\in L^p(X), 1 \le p < 2$,
if $f\ast \varphi_\lambda$ is in $L^1(|\hc(\lambda)|^{-2}\, d\lambda)$, then for almost every $x\in X$,
\[ f(x) = \int_{\mathfrak a^\ast} f\ast \varphi_\lambda(x) |\hc(\lambda)|^{-2}\, d\lambda.\]
In particular, if  $f$ is a $K$-invariant function on $X$ and $\what{f}\in L^1(|\hc(\lambda)|^{-2}\, d\lambda)$,
then\[ f(x) = \int_{\mathfrak a^\ast} \what{f}(\lambda) \varphi_\lambda(x) |\hc(\lambda)|^{-2}\, d\lambda.\]
\end{theorem}

\subsection{Herz's majorizing principle} \label{subsec-Herz} We have the following result due to Herz (\cite{Herz}) on convolution operators.
\begin{proposition} \label{prop-Herz}
Let $h$ be a $K$-biinvariant function on $G$, and let $T_h: f\mapsto f\ast h$ be the corresponding right convolution operator on $L^p(X), p\in [1, \infty]$.
Then the operator norm of $T_h:L^p(X)\to L^p(X)$ obeys the following bound:
\[\|T_h\|_{L^p \rightarrow L^p}  \leq \widehat{h}(-i \gamma_p {\pmb{\rho}})\] where the equality holds if $h$ is nonegative.
\end{proposition}

\subsection{Fourier multipliers} We recall that for $1<p<\infty$,  $\gamma_p=|2/p-1|$ and $\gamma_1=\gamma_\infty=1$. In this paper we are concerned about the bounded linear operators on $L^p(X), 1\le p<\infty$ to itself which are invariant under translations by elements of $G$. This class of operators are called $L^p$-Fourier multipliers or simply $L^p$-multipliers and are denoted by $CO_p(X)$. It is known that $CO_p(X)$ is a Banach algebra. We shall briefly discuss the main points about these operators, collecting them  mostly from \cite{Anker-mult}.  If $T\in CO_1(X)$ then $Tf=f\ast \mu$ where $\mu$ is a $K$-biinvariant finite Borel measure on $G$ and if $T\in CO_2(X)$ then for $f\in C_c^\infty(X)$,
\begin{equation}\wtilde{Tf}(\lambda, k)= m(\lambda) \wtilde{f}(\lambda, k), \label{multiplier-defn}\end{equation} where $m$ is a $W$-invariant function in $L^\infty(\mathfrak a^\ast)$. By abuse of terminologies the function $m(\lambda)$ will also be called a Fourier multiplier. For $1\le  p_1, p_2<\infty$ with  $\gamma_{p_1}\ge \gamma_{p_2}$, $CO_{p_1}(X)\subseteq CO_{p_2}(X)$. In particular $CO_p(X)\subseteq CO_2(X)$ for $1\le p<\infty$ and hence they are also given by \eqref{multiplier-defn} for $f\in C_c^\infty(X)$.
But for $1 \le p<\infty, p\neq 2$, $m(\lambda)$ extends to a $W$-invariant bounded holomorphic function on $\Lambda_p^\circ$. For $p=1$, $m(\lambda)$ is also bounded continuous on $\Lambda_1$.
Henceforth we shall call a multiplier $T\in CO_p(X)$ nontrivial if it is not a constant multiple of the identity operator.

We fix a $p$ in the range $(2, \infty)$ and take a nontrivial $T\in CO_{p}(X)$.     Suppose that  $T$ is   given by the  function   $m(\lambda)$ which by definition is $W$-invariant and extends to a bounded holomorphic function on  $\Lambda_p^\circ$. We have the following result for such $p, T$.
\begin{proposition} \label{prelim-prop-mult} Let $T^\ast: L^{p'}(X)\to L^{p'}(X) $ be the adjoint operator. Then,

\begin{enumerate}
\item[(i)] for any $g\in C_c^\infty(X)$,
$\wtilde{T^\ast g}(\lambda, k)=\overline{m(\bar{\lambda})}\wtilde{g}(\lambda, k)$,  for almost every  $(\lambda, k)\in \Lambda_p^\circ \times K/M$,
$\what{T^\ast g}(\lambda)=\overline{m(\bar{\lambda})}\what{g}(\lambda)$,  for almost every  $\lambda\in \Lambda_p^\circ$,

\item[(ii)] for $\lambda\in \Lambda_p^\circ$, $T\varphi_\lambda=m(\lambda) \varphi_\lambda$  and  $T^\ast \varphi_\lambda=\overline{m(\bar{\lambda})} \varphi_\lambda$.
\end{enumerate}
\end{proposition}
\begin{proof}
We take, $f, g\in C_c^\infty(X)$. Then using the definition of $T^\ast$ and the Plancherel theorem we have,
 \begin{eqnarray*}
\langle T^\ast g, f\rangle=\langle g,  Tf\rangle &=& \int_X g(x) \overline{Tf(x)} dx\\
 &=&\int_{\mathfrak a_+^\ast\times K/M} \wtilde{g}(\lambda, k) \overline{\wtilde{Tf}(\lambda, k)} d\mu(\lambda) dk\\
 &=&\int_{\mathfrak a_+^\ast\times K/M}\overline{m(\lambda)} \wtilde{g}(\lambda, k) \overline{\wtilde{f}(\lambda, k)} d\mu(\lambda) dk.
 \end{eqnarray*}
Since $\overline{m(\lambda)}$ is bounded,  $\overline{m(\lambda)}\wtilde{g}(\lambda, k)\in L^2(\mathfrak a_+^\ast\times K/M)$ and hence there exists unique $\phi\in L^2(X)$ such that $\wtilde{\phi}(\lambda, k)=\overline{m(\lambda)}\wtilde{g}(\lambda, k)$. Therefore $\int_X T^\ast g(x) \overline{f(x)} dx=\int_X \phi(x)\overline{f(x)} dx$ which implies $T^\ast g=\phi$ and in particular $\wtilde{T^\ast g}(\lambda, k)=\overline{m(\lambda)}\wtilde{g}(\lambda, k)$ for all $(\lambda, k)\in \mathfrak a^\ast_+ \times K/M$. Since $T$ is a $p-p$ operator, by duality $T^\ast$ is $p'-p'$. Hence $\lambda\mapsto \overline{m(\lambda)}$  defined on $\mathfrak a^\ast_+$ extends to a holomorphic function on $\Lambda_p^\circ$. As  $m(\lambda)$ also extends as  a holomorphic function on $\Lambda_p^\circ$, we conclude that the extension of $\overline{m(\lambda)}$ is given by $\overline{m(\bar\lambda)}$. This proves the first part of  (i). Integrating both sides of it over $K/M$ we get the second result of (i).

 We recall that  $\varphi_\lambda\in L^p(X)$ for  $\lambda\in \Lambda_p^\circ$ and  $\overline{\varphi_\lambda(x)}=\varphi_{-\bar{\lambda}}(x)$. For a function  $g\in C_c^\infty(X)$ we have,
\[\langle T \varphi_\lambda, g \rangle=\langle \varphi_\lambda, T^\ast g \rangle=\int_G\varphi_\lambda(x) \overline{T^\ast g(x)}dx
=\overline{\int_G\overline{\varphi_\lambda(x)} T^\ast g(x) dx}
=\overline{\int_G\varphi_{-\bar{\lambda}}(x) T^\ast g(x) dx}\]
Therefore,
\[\langle T \varphi_\lambda, g \rangle=\overline{\what{T^\ast g}(\bar{\lambda})}
=\overline{\overline{m(\lambda)} \what{g}(\bar{\lambda})}
=m(\lambda)\overline{\int_G g(x)\varphi_{-\bar{\lambda}} dx}
=m(\lambda)\int_G \varphi_\lambda(x) \overline{g(x)}dx
=\langle m(\lambda)\varphi_\lambda, g\rangle. \] It can be verified in a similar way that $\langle T^\ast \varphi_\lambda, g\rangle=\langle\overline{m(\bar{\lambda})} \phi_\lambda, g\rangle$. Thus   \[T\varphi_\lambda=m(\lambda) \varphi_\lambda \text{ and }T^\ast \varphi_\lambda=\overline{m(\bar{\lambda})} \varphi_\lambda. \hfill\qedhere\]
\end{proof}

\section{Statement and proof of the main result} The following theorem is the main result in this paper.
\begin{theorem} \label{result-multiplier} Fix  $p\in (2, \infty)$. Let $T$ be a nontrivial   $L^{p}$-multiplier. Then there is a constant $c>0$ such that $zT$ for any $z\in \C$ with $|z|=c$ is chaotic on $L^{p}(X)$.  \end{theorem}
Suppose that $T$ is densely defined by $\wtilde{Tf} (\lambda, k)=m(\lambda)\wtilde{f}(\lambda, k)$ for $f\in C_c^\infty(X)$, $\lambda\in \mathfrak a^\ast, k\in K/M$. As $T$ is nontrivial $m(\lambda)$ is a nonconstant function. Then  $|m(\lambda)|$ is also  nonconstant. Indeed if $|m(\lambda)|=\beta$  for some $\beta>0$ for all $\lambda\in \Lambda_p^\circ$ then the holomorphic function  $\lambda\mapsto m(\lambda)$ maps the open domain $\Lambda_p^\circ$ to  an arc of the circle of radius $\beta$, which is  not open in $\C$, which  violates the  open mapping theorem (see Section 2).

As $|m(\lambda)|$ is not constant there exist points $\lambda_1, \lambda_2\in \Lambda_p^\circ$ and $\alpha>0$ such that
\[|m(\lambda_2)|<\alpha<|m(\lambda_1)|.\] Let $c=\frac 1\alpha$. We take a $z\in \C$ such that  $|z|=c$ and   define $m_1(\lambda)= z m(\lambda)$. Then $|m_1(\lambda_2)|<1<|m_1(\lambda_1)|$.
Let  $T_1= zT$. Then $T_1$ is a $L^{p}$-multiplier with symbol $m_1(\lambda)$. The proof of Theorem \ref{result-multiplier} will be completed if we show  that $T_1$ is chaotic on $L^{p}(X)$. This will be done through the next two propositions.

\begin{proposition} \label{hypercyclic-main-result} The operator $T_1$ described above is hypercyclic on $L^p(X)$  for $2<p<\infty$. \end{proposition}
\begin{proof}
As $\lambda\mapsto |m_1(\lambda)|$ is continuous, there exist neighbourhoods $N_1$ and $N_2$ of $\lambda_1, \lambda_2$ respectively in $\Lambda_p^\circ$ such that $|m_1(\lambda)|>1$ for $\lambda\in N_1$ and $|m_1(\lambda)|<1$ for $\lambda\in N_2$. Since $m_1(\lambda)$ is $W$-invariant, we can and will assume that $N_1$ and $N_2$ are subsets of
\[(\Re \Lambda_p^\circ)_+=\{\lambda\in \Lambda_p^\circ \mid \Re \lambda \in \mathfrak a_+^*\}.\] We define
\[Y_1=\mathrm{span}\,\, \{\ell_y\varphi_\lambda\mid \lambda\in N_1, y\in G\} \text{ and } Y_2=\mathrm{span}\,\, \{\ell_y\varphi_\lambda\mid \lambda\in N_2, y\in G\}. \] Both $Y_1$ and $Y_2$ are dense in $L^{p}(X)$. Indeed if any $f\in L^{p'}(X)$ annihilates $Y_1$, then $f\ast \varphi_{-\lambda}\equiv 0$ for $\lambda$ in the open set $N_1$. Since for every fixed $x\in X$, $\lambda\mapsto f\ast \varphi_{-\lambda}(x)$ is holomorphic on $\Lambda_p^\circ$ we have $f\ast \varphi_\lambda\equiv 0$ for all $\lambda\in \Lambda_p$. Using Theorem \ref{St-Tom79} we conclude that $f=0$. Similar argument with the substitution of  $N_1$ by $N_2$  establishes that  $Y_2$ is also dense in $L^p(X)$.

Let \[\eta_\lambda=a_1^\lambda \ell_{y_1^\lambda} \varphi_\lambda+\cdots+a_n^\lambda \ell_{y_n^\lambda} \varphi_\lambda\in Y_1\] be
a finite linear combination of $\ell_y\varphi_\lambda$ with same $\lambda$. We define an operator  $T_1'$ initially on such $\eta_\lambda$ as
 \[T_1'(\eta_\lambda)=m_1(\lambda)^{-1}\eta_\lambda.\] Since  elements of $Y_1$ are finite linear combinations of these $\eta_\lambda$ we  extend $T_1'$ linearly on $Y_1$. We need to  show that $T_1'$ is well defined on $Y_1$.  For future use  we record here that
 \[\eta_\lambda(e)=a_1^\lambda  \varphi_{-\lambda}(y_1^\lambda)+\cdots+a_n^\lambda  \varphi_{-\lambda}(y_n^\lambda).\]
  Let $\xi=\sum_{i=1}^n b_i \eta_{\lambda_i}$ be a typical element of $Y_1$, where $\lambda_i, i=1, \ldots, n$  are distinct. It suffices to  show that if $\xi=0$ then  $\eta_{\lambda_i}=0$ for all $i=1, \ldots, n$, so that $T_1'(\xi)=\sum_{i=1}^n b_i T_1'(\eta_{\lambda_i})=0$.

  Since $N_1\subset (\Re\Lambda_p^\circ)_+$ we note that   $w\lambda_i\neq \lambda_j$  for all nontrivial $w\in W$ whenever $i\neq j$.
Consequently, $\varphi_{\lambda_1}, \varphi_{\lambda_n}$ are two distinct $K$-invariant elements of $L^{p'}(X)$. Therefore there is a $K$-invariant function  $f\in L^p(X)$ such that $\what{f}(-\lambda_1)\neq 0$ and $\what{f}(-\lambda_n)= 0$.
 Starting from   $\xi=0$ and noting that
 $\int_X f(z)\eta_{\lambda_i}(z) dz =\what{f}(-\lambda_i) \eta_{\lambda_i}(e)$,  we get by abuse of notation,
 \[\sum_{i=1}^{m} b_i \what{f}(-\lambda_i) \eta_{\lambda_i}(e)=\sum_{i=1}^{m} c_i  \eta_{\lambda_i}(e)= 0\] for some $m<n$. Indeed, if for any $i=2,\ldots, n-1$, $\what{f}(-\lambda_i)=0$ we discard it and  for others write  $c_i=b_i \what{f}(-\lambda_i)\neq 0$ and relabel them as $i=2, \ldots, m$, keeping $\lambda_1$ unchanged.
 The assumption  $\xi=0$ also implies $\ell_x \xi=0$ for any $x\in G$. Instead of $\xi=0$ if we start from $\ell_{x^{-1}} \xi=0$  then   through the same steps as above, we get
\[\sum_{i=1}^{m} c_i  \eta_{\lambda_i}(x)= 0.\] Thus $\sum_{i=1}^{m} c_i  \eta_{\lambda_i}(x)= 0$ for all $x\in G$.
In this way we can reduce the number of $\eta_{\lambda}$s.  A repeated application of this process  finally yields  $\eta_{\lambda_1}(x)= 0$ which was the target. Thus we have established that $T_1'$ is a well defined operator on $Y_1$.

We shall now verify that operators $T_1$ and $T_1'$ satisfy the hypothesis of  Theorem \ref{kitai}.
Clearly $(T_1')^n\phi\to 0$  as $n\to \infty$ for any $\phi\in Y_1$ because $|m_1(\lambda)|>1$ for $\lambda\in N_1$.
On the other hand as $T_1(\ell_y\varphi_\lambda)=\ell_y T_1(\varphi_\lambda)=m_1(\lambda)\ell_y\varphi_\lambda$ and on $N_2$, $|m_1(\lambda)|<1$,
$(T_1)^n\phi\to 0$  as $n\to \infty$ for any $\phi\in Y_2$.
Lastly,  $T_1T_1'(\ell_y\varphi_\lambda)=\ell_y\varphi_\lambda$ by Proposition 3.5.1 and hence  $T_1T_1'$ is identity on  $Y_1$.
Theorem \ref{kitai} now shows that  $T_1$ is hypercyclic.
\end{proof}

\begin{proposition} \label{periodic-main-result} The set of periodic points of the operator $T_1$ defined above  is  dense  in $L^p(X)$ for $2<p<\infty$. \end{proposition}
\begin{proof}
As there exists $\lambda_1, \lambda_2\in \Lambda_p^\circ$ such that $|m_1(\lambda_1)|<1<|m_1(\lambda_2)|$  it follows from continuity of $m_1$ that there exists  $\lambda_0\in \Lambda_p^\circ$ such that $|m_1(\lambda_0)|=1$. Let $S=\T\cap m_1(\Lambda_p^\circ)$, where $\T$ is the unit circle in the complex plane. By the open mapping theorem (see Section 2)) $m_1(\Lambda_p^\circ)$ is  an open set. Since $m_1(\lambda_0)\in S$,  $S$ is a nonempty open set of $\T$.
We note that $m_1(\Lambda_p^\circ\smallsetminus m_1^{-1}(S))$ is not connected. Indeed, it is union of two nonempty sets, one inside $\T$  and the other outside $\T$
containing $m_1(\lambda_1)$ and $m_1(\lambda_2)$ respectively. We define,
\[I=\{r\in \R \mid e^{2\pi i r}\in S\} \text { and } Z_r=\{z\in \Lambda_p^\circ \mid m_1(z)=e^{2\pi i r}\} \text{ for } r\in I.\]  Then $m_1^{-1}(S)=\cup_{r\in I} Z_r$.
We consider the following subset of $L^p(X)$:
\[Y_3=\mathrm{span } \,\, \{\ell_y \varphi_z \mid y\in G, z\in Z_\nu,  \nu\in \Q\cap I\}.\] Nonemptiness of $Y_3$ follows trivially from the fact that $m_1(\Lambda_p^\circ)$ is open in $\C$.
 If $\nu=a/b\in \Q$,  ($a, b$ relatively prime integers), then using  $T_1(\ell_y \varphi_z)=m_1(z)\ell_y \varphi_z$,  we have  for $z\in Z_\nu$,
\[T_1^b(\ell_y \varphi_z)=m_1(z)^b\ell_y \varphi_z = e^{2\pi a i}\ell_y \varphi_z=\ell_y \varphi_z.\] Thus  the elements  of $Y_3$ are periodic points of $T_1$.

It  remains to show that $Y_3$ is dense in $L^p(X)$. Suppose that a nonzero function $f\in L^{p'}(X)$ annihilates $Y_3$. That is $f\ast \varphi_{-z}(x)=0$ for all $x\in X$ and for all $z\in Z_\nu$,   $\nu\in \Q\cap I$. For a fixed $x\in X$ we define $F_x(z)=f\ast \varphi_{-z}(x)$ for $z\in \Lambda_p^\circ$. Then $F_x$ is holomorphic on $\Lambda_p^\circ$ which vanishes on $\cup_{\nu\in \Q\cap I} Z_\nu$. We claim that  $F_x$ vanishes identically on $\Lambda_p^\circ$.  For the sake of meeting a contradiction we assume that $F_x\not\equiv 0$ on $\Lambda_p^\circ$. Since $F_x$ vanishes on $\cup_{\nu\in \Q\cap I} Z_\nu$,
Lemma \ref{rational-to-real} implies that   $F_x$ vanishes  on the set $m_1^{-1}(S)=\cup_{r\in  I} Z_r$. But as we have assumed that $F_x$ is  a nonzero holomorphic function on $\Lambda_p^\circ$,  $m_1^{-1}(S)$ is a thin set in $\Lambda_p^\circ$. Therefore  by  the properties of thin sets (see Section 2) we conclude that  the set $\Lambda_p^\circ\smallsetminus m_1^{-1}(S)$ is connected. Since $m_1$ is continuous this implies that $m_1(\Lambda_p^\circ\smallsetminus m_1^{-1}(S))$ is connected,    which contradicts our early observation in this proof. Thus $F_x\equiv 0$ on $\Lambda_p$ for all $x\in X$, that is
$f\ast \varphi_\lambda\equiv 0$ on $X$ for all $\lambda\in \Lambda_p$. From this and  Theorem \ref{St-Tom79} we conclude that $f=0$, which establishes that $Y_3$ is dense.
\end{proof}
The following lemma will complete the proof above. We shall use the notation $I$ and $Z_r$ defined in the proof of the proposition above.
\begin{lemma} \label{rational-to-real} Let $I$ and $Z_r$ be as defined in the proof of Proposition {\em \ref{periodic-main-result}}. Fix an $r\in I$. Then for any  $w\in Z_r$  and  $\delta>0$,  there is a $\nu\in I\cap \Q$ and a  $z\in Z_\nu$ such that $|w-z|<\delta$.
\end{lemma}
\begin{proof}
Take the open  ball $B_{\delta'}(w)\subset \Lambda_p^\circ$ where $\delta' <\delta$. Then $w\in B_{\delta'}(w)$. By open mapping theorem
$m_1(B_{\delta'}(w))$ is an open set in   $m_1(\Lambda_p^\circ)$  containing the point $m_1(w)=e^{2\pi i r}$.  So $m_1(B_{\delta'}(w))$ will contain an arc $\{e^{2\pi i s}\mid s\in (a, b)\subset I\}$ with $r\in (a, b)$. Take a $\nu\in (a, b)\cap \Q$. Then the point $e^{2\pi i \nu}$ has a pre-image $z$ in $B_{\delta'}(w)$.   That is  $m_1(z)=e^{2\pi i \nu}$, and hence $z\in Z_\nu$. Also as $z\in  B_{\delta'}(w)$, $|w-z|<\delta'<\delta$.
\end{proof}

\section{Examples and Remarks} Well known examples of Fourier multipliers are spectral multipliers and convolution with suitable Borel measures. In the light of the result proved in the previous section, we shall revisit their dynamics, which will yield some interesting corollaries. The first example also relates Theorem \ref{result-multiplier} with the previous works in this direction e.g. \cite{J-W, Pram-Sar}.

\example  The heat kernel $h_t$ on $X$ for $t >0$ is defined as a $K$-invariant function in the Harish-Chandra $L^p$-Schwartz space $C^p(X)$, $1 \leq p \leq 2$, whose spherical Fourier transform is prescribed as follows (see \cite[3.1]{Ank-Ji}), \[ \widehat{h}_t(\lambda) = e^{-t  ((\lambda, \lambda) + |{\pmb{\rho}}|^2)} \quad \text{ for all } \lambda \in \mathfrak a^\ast.\] For a fixed $t>0$, we consider the operator $Tf= f\ast h_t$,  i.e. $T=e^{-t\Delta}$ where $\Delta$ is the positive Laplace-Beltrami operator on $X$. Then  $m(\lambda)=\what{h_t}(\lambda)$. It is clear that  $T$ is not chaotic on $L^{p}(X)$ for any $1\le p\le \infty$ since $\|T\|_{L^p-L^p}=\what{h_t}(-i\gamma_p\pmb{\rho})=e^{-4t |\pmb\rho|^2/pp'} \le 1$. In general for a multiplier  given by the function $m(\lambda)$, if we define  $\theta=\inf_{\lambda\in \Lambda_p}|m(\lambda)|, \Theta=\sup_{\lambda\in \Lambda_p}|m(\lambda)|$, then it is clear from the proof of Theorem \ref{result-multiplier}, that  we can choose $z$ from the annulus: $1/\Theta<|z|<1/\theta$ where we take $1/\theta=\infty$ if $\theta=0$. Coming back to the case in hand $Tf=f\ast h_t$, we see that  $\theta=0$ and $\Theta=e^{-4t|\pmb\rho|^2/pp'} \le 1$. So we choose $z\in \C$ such that $1 \le e^{4t|\pmb\rho|^2/pp'}<|z|<\infty$. Take $z_0=a+ib$ where $a> 4|\pmb\rho|^2/pp'$ and $b\in \R$. Then $|e^{z_0t}|=e^{at}>e^{4t|\pmb\rho|^2/pp'}$. Thus we can take $z=e^{z_0t}$ and by  Theorem \ref{result-multiplier} $zT=e^{-t(\Delta-z_0)}$ is chaotic on $L^p(X), 2<p<\infty$. A continuous semigroup version of this result is proved in \cite{Pram-Sar}.

\example We continue to use the notation $\theta, \Theta$ defined in the previous example. We  consider convolution by  a nonatomic and nonnegative $K$-invariant  measure $\mu$ on  $X$ such that $\what{\mu}(-i\gamma_p\pmb\rho)<\infty$ for some $1\le p\le \infty$.  By Herz's majorizing principle (see subsection \ref{subsec-Herz}) the operator $Tf= f\ast \mu$ is an $L^p$-multiplier. We note that in this case $\theta<1$, because on $\lambda\in \mathfrak a^\ast$,  $|\what{\mu}(\lambda)|\to 0$ as $|\lambda|\to \infty$. Indeed for $\lambda\in \mathfrak a^\ast$,
\[|\what{\mu}(\lambda)|\le \int_X |\varphi_\lambda(x)| d\mu(x)\le \int_X \varphi_{0}(x) d\mu(x)\le  \int_X \varphi_{i\gamma_p\pmb\rho}(x) d\mu(x)<\infty.\]
Since, for every fixed $x\in X$, $|\varphi_\lambda(x)|\to 0$ as $|\lambda|\to \infty$, the result follows from dominated convergence theorem.

We also note that here  $\Theta=\what{\mu}(-i\gamma_p\pmb\rho)=\|T\|_{L^p-L^p}$. Thus if  $\what{\mu}(-i\gamma_p\pmb\rho)\le 1$, then   $T$ is   a contraction and hence  not chaotic.  On the other hand if $T$ is not a contraction, equivalently, if $\what{\mu}(-i\gamma_p\pmb\rho)>1$ then it is chaotic because we can choose $z=1$ as $1/\Theta<1<1/\theta$.
Precisely, we have proved the following.
\begin{corollary}
\label{result-measure}
Fix $2< p< \infty$. Let $\mu$ be a nonatomic $K$-invariant regular nonnegative Borel measure on $X$ and $T: f\mapsto f\ast \mu$.  If  $\what{\mu}(-i\gamma_p\pmb\rho)<\infty$ (equivalently $\|T\|_{L^p-L^p}<\infty$),
 then $T$ is chaotic on $L^{p}(X)$ if and only if $T$ is not a contraction.
\end{corollary}

\begin{corollary}
\label{result-measure-2} Let  $\mu$ and $T$ be as in Corollary \ref{result-measure}  and  $2<p_2<p_1<\infty$. Suppose that $\mu$ satisfies the condition $\what{\mu}(-i\gamma_{p_1}\pmb\rho)<\infty$, so that  $T\in CO_{p_1}(X)\subset CO_{p_2}(X)$.  If $T$ is chaotic on $L^{p_2}(X)$, then $T$ is chaotic  on $L^{p_1}(X)$.
\end{corollary}
\begin{proof} 
Since $T$ is chaotic on $L^{p_2}(X)$, by Corollary \ref{result-measure} $\what{\mu}(-i\gamma_{p_2}\pmb\rho)>1$. Therefore by the maximum modulus principle (see Section 2) $\what{\mu}(-i\gamma_{p_1}\pmb\rho)>1$ and hence again by Corollary \ref{result-measure}, $T$ is chaotic on $L^{p_1}(X)$.
\end{proof}

Instead of nonnegative $K$-biinvariant measure we can  take a $K$-invariant complex measure, in particular a $K$-invariant measurable function $g$ on $X$ such that
\[1<|\what{g}(-i\gamma_p\pmb\rho)|<\what{|g|}(-i\gamma_p\pmb\rho)<\infty.\] Then the convolution operator $T:L^{p}(X)\to L^{p}(X), 2<p<\infty$ given by $f\mapsto f\ast g$  is bounded and is chaotic on $L^{p}(X)$ by similar argument.

\example Fix a $p$ in the range $2<p<\infty$. From the proof of Theorem \ref{result-multiplier}, it is clear that if an $L^p$-multiplier  $T$ given by the function  $m(\lambda)$ is such that there exists $\lambda_1, \lambda_2\in \Lambda_p^\circ$ with $|m(\lambda_1)|<1<|m(\lambda_2)|$, then $T$ is itself chaotic on $L^p(X)$. For an element  $z\in \C$ from  the complement of the $L^p$-spectrum of the Laplace-Beltrami operator $\Delta$, we  consider the resolvent $T=(\Delta-z)^{-1}$. It is easy to verify that if $z$ is sufficiently close to the spectrum (so that there are $\lambda_1,\lambda_2\in \Lambda_p^\circ$ with  $|(\lambda_1, \lambda_1)+|\pmb\rho|^2-z|<1$   and $|(\lambda_2, \lambda_2)+|\pmb\rho|^2-z|>1$) then $T$ is chaotic. A description of the $L^p$-spectrum of $\Delta$ can be found in  \cite[3.4]{Pram-Sar} and the references therein.

\section{Sharpness of the range of $p$} Aim of this section is to  show that the range of $p$ in Theorem \ref{result-multiplier} is  sharp. As a preparation we gather and prove some lemmas. The first one is  from \cite[Proposition 5.1]{GKPA}.
\begin{lemma} \label{not-hyper-1} Let $T$ be a hypercyclic operator  on a Banach space $\B$ and $T^\ast$ be the dual operator of $T$ acting on $\B^\ast$. Then
\begin{enumerate}
\item[(i)] for any nonzero $\phi\in \B^\ast$ the orbit
$\{(T^\ast)^n\phi \mid n\ge 0\}$ is unbounded,
\item[(ii)] the point spectrum of $T^\ast$ is empty.
\end{enumerate}
\end{lemma}
An easy  adaptation of \cite[Theorem 8.1]{Helga-Abel} using  the fact  that $\varphi_\lambda\in L^{p'}(X)$ for $\lambda\in \Lambda_q$ (see Lemma \ref{lemma-eigenfns} (e)), proves the following lemma. See also \cite{HRSS, Sar-Sita-chennai}.
\begin{lemma}\label{HFT-exist}
For $1\le p< q<2$ and $f\in L^p(X)$,  there exits a subset $B\subset K/M$ of full measure  such that for each $k\in B$,  $\wtilde{f}(\lambda, k)=\int_X f(x) e^{(i\lambda-\pmb\rho) H(x^{-1}k)} dx$ exists for all $\lambda\in \Lambda_q$ and is holomorphic on $\Lambda_q$. The set $B$ may depend on the function $f$ but does not depend on $\lambda\in \Lambda_q$.
\end{lemma}

We also have the following results.
\begin{lemma} \label{restriction} For $f\in L^p(X), 1\le p<2$ and $\lambda\in \mathfrak a^\ast$, $\|\wtilde{f}(\lambda, \cdot)\|_{L^2(K/M)}\le C \|f\|_p$  for some constant $C>0$.
\end{lemma}
\begin{proof}
Temporarily using the notation $e_{\lambda, k}(x)= e^{(i\lambda-\pmb\rho) (H(x^{-1}k))}$ we have,
\begin{eqnarray*} \int_{K/M} |\wtilde{f}(\lambda, k)|^2 dk &=& \int_{K/M} \overline{\wtilde{f}(\lambda, k)}\,\, \wtilde{f}(\lambda, k) dk\\
&=& \int_{K/M} \int_X \overline{f(x)} \overline{e_{\lambda, k}(x)} dx \wtilde{f}(\lambda, k) dk\\
&=& \int_X \overline{f(x)} \int_{K/M}  \overline{e_{\lambda, k}(x)}  \wtilde{f}(\lambda, k) dk dx\\
&=& \int_X \overline{f(x)} f\ast \varphi_\lambda(x) dx\\
&\le & \|f\|_p\,\, \|f\ast \varphi_\lambda\|_{p'},
\end{eqnarray*} where in the last step we have used  H\"older's inequality.
We recall that for $\lambda\in \mathfrak a^\ast$,   $\varphi_\lambda\in L^{2+\epsilon}(X)$ for any $\epsilon>0$ (see Lemma \ref{lemma-eigenfns} (e)). This implies that  the operator $f\mapsto f\ast \varphi_\lambda$ is bounded from $L^p(X)$ to $L^{p'}(X)$ for any $1 \le p<2<p' \le\infty$ (see \cite[Theorem 2.2]{Cow-Meda-Guil-1}). That is $\|f\ast \varphi_\lambda\|_{p'}\le C\|f\|_p$ for $\lambda\in \mathfrak a^\ast$ for some constant $C>0$ and $1\le p<2$.
Therefore we have $\int_{K/M} |\wtilde{f}(\lambda, k)|^2 dk \le  C \|f\|_p^2$, which is the assertion.
\end{proof}
\begin{lemma} \label{multiplier-factors} For $1\le p<q<2$, let $T$ be an $L^p$-multiplier given by the function $m(\lambda)$ and $f\in L^p(X)$. Then there exits a subset $B\subset K/M$ of full measure  such that for each $k\in B$ and for  $\lambda\in \Lambda_q$,  $\wtilde{Tf}(\lambda, k)=m(\lambda) \wtilde{f}(\lambda, k)$.
\end{lemma}
\begin{proof}
 Using the denseness of $C_c^\infty(X)$ in $L^p(X)$, we find a sequence $f_n\in C_c^\infty(X)$ which converges to $f$ in $L^p(X)$. Then passing to a subsequence $f_{n_i}$ if necessary, we have  by Lemma \ref{restriction} \[\wtilde{f_{n_i}}(\lambda, k)\to \wtilde{f}(\lambda, k)\] for every fixed $\lambda\in \mathfrak a^\ast$ for almost every $k\in K/M$.
 We also have $Tf_{n_i}\to Tf$ in $L^p$ and hence for a finer subsequence $\wtilde{Tf_{n_{i_k}}}(\lambda, k)\to  \wtilde{Tf}(\lambda, k)$ for  every fixed $\lambda\in \mathfrak a^\ast$ for almost every $k\in K/M$. By definition $\wtilde{Tf_{n_{i_k}}}(\lambda, k)=m(\lambda)\wtilde{f_{n_{i_k}}}(\lambda, k)$ for those $k\in K/M$ and  thus for every fixed $\lambda\in \mathfrak a^\ast$, $\wtilde{Tf_{n_{i_k}}}(\lambda, k)$ converges to $m(\lambda)\wtilde{f}(\lambda, k)$,  for almost every $k\in K/M$. This establishes that for every fixed $\lambda\in \mathfrak a^\ast$, $\wtilde{Tf}(\lambda, k)=m(\lambda)\wtilde{f}(\lambda, k)$,  for almost every $k\in K/M$. We note that we have a set $B\subset K/M$ of full measure in $K/M$, such that for every fixed  $k\in B$,  both $\lambda\mapsto \wtilde{f}(\lambda, k)$  and $\lambda\mapsto \wtilde{Tf}(\lambda, k)$ are holomorphic on $\Lambda_q$. Therefore the equality $\wtilde{Tf}(\lambda, k)=m(\lambda)\wtilde{f}(\lambda, k)$ extends to all $\lambda\in \Lambda_q$ and $k\in B$.
\end{proof}
We are  now ready to show the sharpness of the range of $p$.
\begin{proposition} Fix $1\le p< 2$. Let $T:L^p(X)\to L^p(X)$ be a nontrivial  $L^p$-multiplier. Then $T$ is neither hypercyclic nor it has any periodic point.
\end{proposition}
\begin{proof}
We recall that  every $\varphi_\lambda$ with $\lambda\in \Lambda_p^\circ$ is an eigenfunction  of $T^\ast:L^{p'}(X)\to  L^{p'}(X)$. (See Lemma \ref{lemma-eigenfns} (e) and Proposition \ref{prelim-prop-mult}.) Therefore by Lemma \ref{not-hyper-1} (ii), $T$ is not hypercyclic.

We suppose that the multiplier $T$ is given by the function $m(\lambda)$. We fix a $q\in (p, 2)$. If for a nonzero function  $g\in L^p(X)$,  $T^n g=g$  for some $ n\in \N, n>0$, then by Lemma \ref{HFT-exist} and Lemma \ref{multiplier-factors}, there exits a subset $B\subset K/M$ of full measure  such that for each $k\in B$, $(m(\lambda)^n-1)\wtilde{g}(\lambda, k)=0$ for $\lambda\in \Lambda_q$. Since $\wtilde{g}(\lambda, k)$, for $k\in B$ is holomorphic on $\Lambda_q$ it can be zero on a thin set which has $2n$-dimensional Lebesgue measure zero. Thus $m(\lambda)^n =1$ on $\Lambda_q$, that is $|m(\lambda)|=1$. This is not possible as $m(\lambda)$ is holomorphic and hence an open map.
\end{proof}

\begin{proposition}  Let $T:L^2(X)\to L^2(X)$ be a nontrivial  $L^2$-multiplier. Then $T$ is not hypercyclic and hence not chaotic.
\end{proposition}
\begin{proof}
Let $m\in L^\infty(\mathfrak a^*_+)$ and the operator $T$ is given by $\wtilde{Tf}(\lambda, k)=m(\lambda)\wtilde{f}(\lambda, k)$. Then $\|T\|_{L^2-L^2}=\|m\|_\infty$. We assume for the sake of meeting a contradiction that $T$ is hypercyclic, equivalently there exists a hypercyclic vector $\phi\in L^2(X)$ for $T$. Then there exits a sequence $\{n_k\}$ of natural numbers such that  $T^{n_k} \phi\to 2\phi$ in $L^2(X)$ as $n_k\to \infty$. For convenience by abuse of notation we write $n_k$ as $n$. We have consequently,
$\|T^n\phi\|_2\to 2\|\phi\|_2$, that is
\[\lim_{n\to \infty}\int_{\mathfrak a_+^\ast\times K/M} |m(\lambda)|^{2n} |\wtilde{\phi}(\lambda, k)|^2 |\hc(\lambda)|^{-2} d\lambda\; dk=4 \int_{\mathfrak a_+^\ast\times K/M} |\wtilde{\phi}(\lambda, k)|^2 |\hc(\lambda)|^{-2} d\lambda\; dk.\] We divide  the integral in the left hand side in three parts and apply dominated convergence theorem to get,
\begin{eqnarray*}\lim_{n\to \infty}\int_{\{\lambda\in \mathfrak a_+^\ast\mid |m(\lambda)|>1\}\times K/M} |m(\lambda)|^{2n} |\wtilde{\phi}(\lambda, k)|^2 |\hc(\lambda)|^{-2} d\lambda\; dk\\
+\int_{\{\lambda\in \mathfrak a_+^\ast\mid |m(\lambda)|=1\}\times K/M} |\wtilde{\phi}(\lambda, k)|^2 |\hc(\lambda)|^{-2} d\lambda\; dk\\
=4 \int_{\mathfrak a_+^\ast\times K/M} |\wtilde{\phi}(\lambda, k)|^2 |\hc(\lambda)|^{-2} d\lambda\; dk.
\end{eqnarray*} Thus,
\[\lim_{n\to \infty}\int_{\{\lambda\in \mathfrak a_+^\ast\mid |m(\lambda)|>1\}\times K/M} |m(\lambda)|^{2n} |\wtilde{\phi}(\lambda, k)|^2 |\hc(\lambda)|^{-2} d\lambda\; dk
\le 4 \int_{\mathfrak a_+^\ast\times K/M} |\wtilde{\phi}(\lambda, k)|^2 |\hc(\lambda)|^{-2} d\lambda\; dk. \] By monotone convergence theorem the left hand side goes to infinity while the right hand side is finite. Hence either $\wtilde{\phi}\equiv 0$ on $\{\lambda\in \mathfrak a_+^\ast\mid |m(\lambda)|>1\}\times K/M$ or the set $\{\lambda\in \mathfrak a_+^\ast\mid |m(\lambda)|>1\}\times K/M$ has measure zero in $\mathfrak a_+^\ast\times K/M$. By Plancherel theorem, in the first case $\|T\phi\|_2\le \|\phi\|_2$, hence $\phi$ is not a hypercyclic vector  and in the second $T$ is a contraction. Both of these conclusions  contradict our assumption.
\end{proof}
\begin{remark}
\begin{enumerate}
\item Following \cite[Theorem 1.2]{Pram-Sar} one can  give a different  proof of the fact that an $L^2$-multiplier cannot be hypercyclic. This is based on the observation that $T$ being a  multiplier preserves the left-$K$-types of a function $\phi$. Thus a mismatch between the $K$-types of the possible hypercyclic vector $\phi$ and the target function $f$, will prevent  the sequence $T^n\phi$   to converge to $f$ in $L^2$.
\item An $L^2$-multiplier can have periodic points. Indeed,  there are nontrivial $L^2$-multipliers which have a dense set of periodic points. For instance for a rank one symmetric space $X$  we define a multiplier $T$ by the following   prescription: $m(\lambda)=1$ for $\lambda\in (0,1)$ and $m(\lambda)=-1$ otherwise. Then for any $f\in L^2(X)$, $T^2 f=f$.
\item It is easy to find nontrivial $L^\infty$-multipliers $T$ such that no constant multiple of $T$ is chaotic. For instance, if $f\in L^1(X)$ is a $K$-biinvariant  function, then $T: g\mapsto g\ast f$ cannot be hypercyclic on $L^\infty(X)$. Indeed for any $\phi\in L^\infty(X)$, $\phi\ast f$ is continuous, hence $T^n\phi$ is a sequence of continuous functions. Since its  uniform limit is a continuous function, it cannot converge to an arbitrary function in $L^\infty(X)$.
\end{enumerate}
\end{remark}

\section{Open Questions}
The results in this article triggers some questions, which we offer to the readers. For the sake of simplicity in this section we shall restrict to rank one symmetric spaces where $\Lambda_p$ defined in (3.1.4) takes a simpler form:
\[\Lambda_p=\{\lambda\in \C\mid |\Im \lambda|\le \gamma_p\rho\},\] where $\rho$ is interpreted as a positive number. However the discussion here is equally valid for arbitrary rank.

1. We choose $p_1, p_2$ such that  $2<p_2<p_1<\infty$.  Then  $CO_{p_1}(X)\subset CO_{p_2}(X)$.  It is possible to construct a linear operator  $T\in CO_{p_1}(X)\subset CO_{p_2}(X)$ which is chaotic on $L^{p_1}(X)$ but not chaotic on $L^{p_2}(X)$.  For instance we can take $T=e^{-t(\Delta-c)}$ where $\Delta$ is the positive Laplace-Beltrami operator, $t>0$ and $c$ is a constant satisfying
\[\frac{4|\pmb\rho|^2}{p_1p_1'}<c<\frac{4|\pmb\rho|^2}{p_2p_2'}\] for $p_1, p_2$ as above. Then $T$ will be chaotic on $L^{p_1}(X)$ but not on $L^{p_2}(X)$. To see that $T$ will be chaotic on $L^{p_1}(X)$ we first note that $T$ is given by the symbol $m(\lambda)=e^{-t((\lambda, \lambda)+|\pmb\rho|^2-c)}$, $\lambda\in \Lambda_{p_1}$ and that $\frac{4|{\pmb{\rho}}|^2}{pp'} = \bigl((i \gamma_p)^2 + 1 \bigr) |{\pmb{\rho}}|^2$. Writing $\lambda=u+iv$ where $|v|<\gamma_{p_1}|\pmb\rho|$ we have $|m(\lambda)|=e^{-t(|u|^2-|v|^2+|\pmb\rho|^2-c)}$. The given condition on $c$ implies \[\gamma_{p_1}|\pmb\rho|>\sqrt{|\pmb\rho|^2-c}>\gamma_{p_2}|\pmb\rho|.\] Taking $u$ sufficiently large we have $|m(\lambda)|<1$ for the corresponding $\lambda$.  On the other hand choosing $u=0$ and $v$ in the range $\gamma_{p_1}|\pmb\rho|>|v|> \sqrt{|\pmb\rho|^2-c}$, we get $|m(\lambda)|>1$. The argument in the proof of Theorem \ref{result-multiplier} now shows that $T$ is chaotic on $L^{p_1}(X)$. It is  clear that such a choice is not possible for $L^{p_2}(X)$. Result in \cite[Theorem 1.3]{Pram-Sar} also shows that  $T$ is neither  hypercyclic nor it has {\em any}  periodic point in   $L^{p_2}(X)$.

The operator $T=e^{-t(\Delta-c)}$ considered here is indeed  a  convolution operator by the $K$-invariant measure $e^{ct} h_t$ on $X$. We have shown in Corollary \ref{result-measure-2} that whenever $T\in CO_{p_1}(X)\subset CO_{p_2}(X)$ is convolution by a nonatomic $K$-invariant nonnegative measure $\mu$ on $X$, then $T$ is chaotic on $L^{p_2}(X)$ implies that it is chaotic on $L^{p_1}(X)$.
We are thus led to ask the following question: Let $T\in CO_{p_1}(X)\subset CO_{p_2}(X)$ where  $2<p_2<p_1<\infty$. Suppose that   $T$ is chaotic (respectively hypercyclic) on $L^{p_2}(X)$. Does it follow that $T$ is chaotic (respectively hypercyclic)  on $L^{p_1}(X)$?

2. In Corollary \ref{result-measure} we have shown  that if $T:f\mapsto f\ast\mu$ is a convolution operator initially defined for $f\in C_c^\infty(X)$, where $\mu$ is a nonatomic $K$-invariant measure on $X$ which satisfies $\what{\mu}(-i\gamma_p\pmb\rho)<\infty$, for some $p\in (2, \infty)$, then $T\in CO_p(X)$ and it is either a contraction (when $\what{\mu}(-i\gamma_p\pmb\rho)\le 1$) or it is chaotic (when $\what{\mu}(i\gamma_p\pmb\rho)> 1$).
This motivates us to ask  the following question: Let $T\in CO_{p}(X)$ for some $2<p<\infty$ be a nontrivial multiplier on $L^p(X)$ which is not a contraction. Is $T$ chaotic?

A related  question motivated by the same (i.e. convolution with noatomic $K$-invariant measure on $X$) is the following:  Let $T:L^p(X)\to L^p(X)$ for some $2<p<\infty$ be a  $L^p$-multiplier given by the symbol $m(\lambda)$. If $|m(\lambda)|\le 1$ on $\Lambda_p^\circ$, then is it true  that $T$ is not hypercyclic?

3.
Let $T\in CO_{p_1}(X)$ be a nontrivial multiplier with symbol $m(\lambda)$ for  some $2<p_1<\infty$. Then $T\in CO_p(X)$ for all $p\in [2, p_1]$. We note  that $|m(\lambda)|$ is nonconstant on any open set of $\Lambda_{p_1}^\circ$. Therefore for any  $\delta>0$ such that $2+\delta<p_1$, $|m(\lambda)|$ is nonconstant on $\Lambda_{2+\delta}$. The argument of the proof of Theorem \ref{result-multiplier} shows that $zT$ is chaotic on $L^p(X)$ for any $p\in [2+\delta, p_1]$ if we can choose two elements $\lambda_1, \lambda_2\in \Lambda_{2+\delta}^\circ$ such that $z\in \C$ satisfies $|m(\lambda_1)|<1/|z|<|m(\lambda_2)|$. This argument however prevents us to make a uniform choice  for the whole range $[2, p_1]$, which can be illustrated through the following example in  a  rank one symmetric space $X$. We define a multiplier operator $T$ by  $m(\lambda)=e^{i/(4\pmb\rho^2+\lambda^2)}$ for $\lambda\in \Lambda_{p_1}$. It can be verified that $T\in CO_{p_1}(X)$ (see \cite{Anker-mult}). Since $|m(\lambda)|=1$ on $\mathfrak a^\ast=\R$, we cannot choose $\lambda_1, \lambda_2$ from $\R$ satisfying $|m(\lambda_1)|<|m(\lambda_2)|$ and proceed as above. Thus the question remains whether  it is possible to find a constant $c>0$ such that  for all $z\in \C$ with $|z|=c$,  $zT$  is chaotic on $L^p(X)$ for all $p\in [2, p_1]$.

\end{document}